\documentclass[11pt,a4paper]{article}
\usepackage{xypic,a4wide,amsmath,amssymb,amsthm,mathtools,xcolor}
\mathtoolsset{showonlyrefs}
\usepackage[mathcal]{euscript}
\usepackage{mathrsfs}
\usepackage{enumitem}
\let\mathcal=\mathscr

\newtheorem{thm}{Theorem}[section]
\newtheorem{corol}[thm]{Corollary}

\newtheorem{lemma}[thm]{Lemma}
\newtheorem{conj}[thm]{Conjecture}

\newtheorem{defin}[thm]{Definition}

\theoremstyle{remark}
\newtheorem{rem}[thm]{Remark}
\newtheorem{ex}[thm]{Example}

\newcommand{\End}{\mbox{\it End}\,}

\newcommand{\cO}{{\mathcal O}}

\newcommand{\E}{\mathfrak E}
\newcommand{\F}{\mathfrak F}
\newcommand{\gS}{\mathfrak S}
\newcommand{\gQ}{\mathfrak Q}
\newcommand{\gG}{\mathfrak G}
\newcommand{\grass}[2]{\operatorname{Gr}_{#1}(#2)}

\newcommand{\rk}{\operatorname{rk}}

\newcommand{\thegrass}{\operatorname{Gr}_d(E)}

\newcommand{\C}{{\mathbb C}}

\newcommand{\Q}{{\mathbb Q}}
\renewcommand{\P}{{\mathbb P}}

\newcommand{\qee}{\mbox{\hspace{0.2mm}}\hfill$\triangle$}

\begin{document}
\begin{center} \bf   ON A CONJECTURE ABOUT HIGGS BUNDLES  \\[3pt] AND SOME INEQUALITIES \end{center}

\thispagestyle{empty} \vspace{-3mm}
\begin{center}{\sc Ugo Bruzzo,$^{abcd}$  Beatriz Gra\~na Otero$^{e}$} and {\sc Daniel Hern\'andez Ru\'iperez}$^{ef}$ \\[5pt]
\small
$^a$ SISSA (International School for Advanced Studies), Via Bonomea 265, 34136 Trieste, Italy \\
$^b$ Departamento de Matem\'atica, Universidade Federal da Para\'iba,  Jo\~ao Pessoa, PB, Brazil \\
$^c$ INFN (Istituto Nazionale di Fisica Nucleare), Sezione di Trieste \\
$^d$ IGAP (Institute for Geometry and Physics), Trieste  \\
$^e$ Departamento de Matem\'aticas  and IUFFYM (Instituto de F\'\i sica Fundamental y \\ Matem\'aticas), Universidad de Salamanca,
Plaza de la Merced 1-4, 37008 Salamanca, Spain \\
$^f$ Real Academia de Ciencias Exactas, F\'isicas y Naturales, Spain\\[5pt] 
Email: {\tt bruzzo@sissa,it, beagra@usal.es, ruiperez@usal.es}
\end{center}

\vfill

\begin{abstract}  We briefly review an open conjecture about Higgs bundles that
are semistable after pulling back to any curve, and prove it in the rank 2 case. We also prove some results in higher rank under suitable additional assumptions.
Moreover, we establish a set of inequalities holding for H-nef Higgs bundles that generalize some of 
the Fulton-Lazarsfeld inequalities for numerically effective vector bundles.
\end{abstract}

\vfill\begin{minipage}{\textwidth} \small
\parbox{\textwidth}{\hrulefill} \\
Date: 30 April 2023. Revised 31 July 2023 
U.B.'s research is partly supported by Bolsa de Produtividade 313333/2020-3 of Brazilian CNPq, by  PRIN ``Birational geometry and moduli problems'' and INdAM-GNSAGA.  B.G.O.~and D.H.R.'s research is partly supported by Grant PID2021-128665NB-I00 funded by MCIN/AEI/ 10.13039/501100011033,  by ``ERDF A way of making Europe,''  and Universidad de Salamanca through Programa XIII.

\end{minipage}

\newpage

\section{Introduction}
The progenitor of the results discussed in this paper may be traced back to a theorem by
Miyaoka   \cite{Mi}, which  characterizes  the semistability of a vector bundle $E$ on a smooth  projective curve $X$
in terms of the nefness of a numerical class in the projectivized bundle $\P E$: if
\begin{equation}\label{lambda}\lambda (E) = c_1(\cO_{\P E}(1)) - \tfrac 1r\pi_1^\ast(c_1(E)) \in N^1(\P E) \otimes \Q,
\end{equation}
where $\pi_1\colon \P E \to X$ is the projection, and $r= \rk E$, then $E$ is semistable if and only if $\lambda(E)$ is nef
(note that $r\lambda(E)$ is the relative anticanonical   class of $\P E$ over $X$).

\subsection{Curve semistable (Higgs) bundles} The following theorem was proved in \cite{Nakayama} and rediscovered in \cite{BHR} in a slightly different and seemingly stronger, albeit equivalent
form. It may be regarded as a higher dimensional generalization of Miyaoka's theorem. Let $X$ be an $n$-dimensional smooth connected
complex projective variety. For any coherent $\cO_X$-module $F$ of positive rank  define its discriminant as
$$\Delta(F) = c_2(F) - \tfrac{\rk F-1}{2\rk F} \, c_1(F)^2 \in H^4(X,\Q).$$
{Moreover, if $E$ is a vector bundle on $X$, the class $\lambda(E)$ is defined as in equation \eqref{lambda}.}

\begin{thm}\label{TeoBO}  Let $E$ be a vector bundle on $X$.  The following conditions are equivalent:
\begin{enumerate} \itemsep=-2pt
   \item $E$ is semistable with respect to some polarization $H$, and $\Delta(E)=0$;
  \item for any morphism $f\colon C\to X$, where $C$ is a smooth  projective curve, the vector bundle $f^* E$ is semistable;
  \item the class $\lambda(E)$ is nef.
\end{enumerate}
\end{thm}

\noindent(In Nakayama the condition on the discriminant was $\Delta(E)\cdot H^{n-2} = 0 $, but via Theorem 2 in \cite{simpson-local} this
is readily shown to be equivalent to $\Delta(E)=0$ whenever $E$ is semistable with respect to $H$.) We shall call {\em curve semistable} the vector bundles
satisfying condition (ii). It may be natural to wonder if Theorem \ref{TeoBO} also holds true for Higgs bundles. We recall that a Higgs sheaf
is a pair $\F=(F,\phi)$, where $F$ is a coherent $\cO_X$-module, and $\phi\colon F \to F\otimes\Omega^1_X$ is an 
$\cO_X$-linear morphism such that the composition 
$$\phi\wedge\phi \colon F \xrightarrow{\phi} F\otimes\Omega^1_X  \xrightarrow{\phi\times\mathrm{id}}   F\otimes \Omega^1_X\otimes\Omega^1_X \to F\otimes\Omega^2_X$$
is zero. A Higgs bundle is a Higgs sheaf with $F$ locally free. Semistability and stability are defined as for vector bundles but only with reference to $\phi$-invariant subsheaves.
Curve semistability is defined as for vector bundles. So the Higgs bundle version of Theorem \ref{TeoBO} {is the following conjecture:}

\begin{conj}\label{TeoHiggs}  Let $\E=(E,\phi)$ be a Higgs bundle on $X$.  The following conditions are equivalent:
\begin{enumerate} \itemsep=-2pt
   \item $\E$ is semistable with respect to some polarization $H$, and $\Delta(E)=0$;
  \item for any morphism $f\colon C\to X$, where $C$ is a smooth   projective curve, the Higgs bundle $f^* \E$ is semistable.
\end{enumerate}
\end{conj}

\noindent (We shall state the condition generalizing the nefness of the class $\lambda(E)$ later on.) {The fact that condition (i) implies condition (ii) was proved in \cite{BHR}. A motivation for expecting that the opposite implication} may hold true is Bogomolov inequality \cite{HL}:
if $E$ is a vector bundle on an $n$-dimensional smooth   projective variety, semistable with respect
to a polarization $H$, 
then $\Delta(E)\cdot H^{n-2} \ge 0$.
The underlying vector bundle $E$ of a semistable Higgs bundle $\E=(E,\phi)$ satisfies the same inequality, even when
$E$ itself is not semistable \cite{Simpson-unif}; i.e., semistability is a sufficient but non-necessary condition for the non-negativity of the quantity $ \Delta(E)\cdot H^{n-2} $, and one can imagine the same happens for the vanishing of $\Delta(E)$ for curve semistable   bundles.

We conjecture
that the reverse implication holds true for any smooth   projective variety. In this paper we prove this 
when $\E$ has rank two; actually, we prove the implication {in any rank} when the Grassmannian of Higgs quotients of some rank (to be defined later) has a component which is a divisor in the full Grassmannian and surjects onto $X$. Then we prove that such a component always exists in rank two.

\subsection{Higgs varieties} 
One easily shows that a curve semistable Higgs bundle is semistable with respect to any polarization. So the nontrivial content of the conjecture is the following {statement}:

\begin{center}\em
A curve semistable Higgs bundle has vanishing discriminant.
\end{center}
{Here curve semistability for Higgs bundles is defined as in condition (ii) of Conjecture \ref{TeoHiggs}.} Waiting for the conjecture to be eventually settled in the positive or negative, it makes sense to
prove it for specific classes of varieties.  
The authors of  \cite{BruLoGiu} defined a 
  \emph{Higgs variety} $X$ as one on which the conjecture holds.  The easiest case is that
  of varieties with slope-semistable cotangent bundle of nonnegative degree, simply
  because in this situation the underlying vector bundle $E$ of a curve semistable Higgs bundle 
  is itself curve semistable. Starting from this one can identify other Higgs varieties, such as:

\begin{itemize} \itemsep=-2pt
\item rationally connected varieties;
\item abelian varieties; 
\item fibrations over a Higgs variety whose fibers are rationally connected;
\item bases of  finite \'etale covers whose total space is a Higgs variety;
\item varieties of dimension $\ge 3$ containing an effective ample divisor which is a Higgs variety;
\item varieties with nef tangent bundle (in dimension 2 and 3 these were classified in \cite{DPS});
\item varieties birational to a Higgs variety.
\end{itemize}

Moreover, 
 in \cite{BLLG} it was shown that   algebraic K3 surfaces
 are Higgs varieties, and this was extended, using different techniques, to simply connected
 Calabi-Yau varieties in \cite{BruCap}. Some results in the case of elliptic surfaces are proved
 in \cite{BruPer1}. A review of this problem updated to 2017 can be found in \cite{BRcong}.

\subsection{Contents} 
The main tool we use in this paper is the {\em Higgs Grassmannian} of a Higgs bundle $\E=(E, \phi)$, a notion that some of us introduced in \cite{BHR}.
This object is defined in Section \ref{HiggsGrass}, where some of its basic properties are studied. It seems quite difficult
to find general results about the Higgs Grassmannian, but its structure is quite clear in the case $\rk E=2$, and this is indeed the
key to the proof of the conjecture in the rank 2 case that we give in Section \ref{conjrk2}. Actually in Section \ref{giapponesi} we prove the conjecture assuming that the rank $d$ Higgs Grassmannian $\grass{d}\E$ has a component that is a divisor in the full Grassmann bundle $\thegrass$ which surjects onto $X$. Such a divisor always exists in the rank 2 case, due to the fact that the Higgs Grassmiannian of a rank 2 Higgs bundle over a curve is never empty, thus providing a full proof of the conjecture in the rank 2 case.

 The Higgs Grassmannian allows one to introduce a notion of numerical effectiveness for Higgs bundles, a notion that ``feels'' the Higgs field.
 This was studied in \cite{bruzzo-grana-adv, bruzzo-grana-crelle}. In the final Section \ref{inequalities} of this paper we show that
 Higgs bundles that are numerically effective in this sense satisfy some inequalities which generalize some of the 
 Fulton-Lazarsfeld inequalities for numerically effective vector bundles  (\cite{FulLaz}, see also \cite{DPS}).

\paragraph{Notation and conventions.} All varieties and schemes are over the complex numbers, and, unless otherwise stated, all varieties are supposed to be connected.
A ``sheaf'' on a scheme $X$ will be a coherent $\cO_X$-module.

\section{The Higgs Grassmannian} \label{HiggsGrass} The {\em Higgs Grassmannian} is an object that parameterizes locally free Higgs quotients of a Higgs bundle
 exactly as the usual Grassmann bundle parameterizes locally free quotients of a vector bundle. This was
 introduced in \cite{BHR}. We recall here its definition and some of its properties.
 
 \subsection{Definition of the Higgs Grassmannian}
 Let $X$ be a smooth  variety over $\C$.
 For a given   rank $r$ vector bundle $E$ on  $X$, and for every $d$ in the range $ 0 < d < r$, 
 we denote the Grassmann bundle of rank $d$ locally free quotients    of $E$ as  $\grass dE$.  
 Since $\grass1E=\P E$ we shall use the latter notation. One has the universal exact sequence  
$$ 0 \to S_d  \to \pi_d^\ast E \to Q_d \to 0 $$ of vector bundles
on $\grass dE$, where $Q_d$ is the rank $d$ universal quotient bundle,   $S_d$ is the corresponding kernel, and   $\pi_d\colon\grass dE \to X$ is the projection.
If $\E=(E,\phi)$ is a Higgs bundle, we form the diagram
\begin{equation}\label{diagram}\xymatrix{0\ar[r]&S_d  \ar[r]^{{a}_d}& \pi^\ast_d   E \ar[r]\ar[d]^{\pi^\ast_d \phi}& Q _d  \ar[r]&0\\
0\ar[r]&S_d  \otimes \Omega^1_{\grass dE}\ar[r]&\pi^\ast_d   E\otimes \Omega^1_{\grass dE}\ar[r]^{{b}_d}& Q _d  \otimes \Omega^1_{\grass dE}\ar[r]&0}
\end{equation} 
The $d$-th Higgs Grassmannian of $\E$, denoted $\grass d{\E}$, is the subscheme of $\grass dE$ defined by the zero locus of the composition 
${b}_d\circ\pi_d^\ast\phi\circ {a}_d$. By construction, the restrictions of the bundles $S_d$ and $Q_d$ to   $\grass d{\E}$ carry Higgs fields
induced by $\pi_d^\ast\phi$, so that we have an exact sequence of Higgs bundles on $\grass d{\E}$ 
$$ 0 \to \gS_d \to \rho_d^\ast \E \to \gQ_d \to 0 .$$
The scheme  $\grass d{\E}$ may be singular, reducible, nonreduced, non-equidimensional. On the positive side it enjoys
the analogous universal property of the usual Grassmann bundles:
if $f\colon Y \to X$ is a scheme morphism, and $\gG$ is a rank $d$  locally free Higgs quotient of
$f^\ast\E$, there is a morphism $g\colon Y \to \grass d{\E}$ such that
$\gG=g^\ast \gQ_d$, and the diagram
$$\xymatrix{
& \grass d{\E} \ar[d]^{\rho_d} \\
Y\ar[ur]^g\ar[r]_f & X}
$$
commutes.

\label{nefthetas} Now assume that $X$ is projective.
Given a rank $r$ Higgs bundle $\E=(E,\phi)$ on $X$, for every $ 0 < d < r$ we define the following classes
in $N^1(\grass d{\E})\otimes\Q$
\begin{equation}\label{thetas} \theta_d(\E) = c_1(Q_d) -  \tfrac dr \rho_d^\ast (c_1(E)).\end{equation}
It was proved in \cite{BHR} (see also  \cite{bruzzo-grana-adv}) that $\E$ is curve semistable if and only if {\em all}
classes $\theta_d(\E)$ are nef. Note that $\theta_1(\E)$ is the restriction of the class $\lambda(E) \in N^1(\P E)\otimes \Q$
to $\grass 1{\E}$. Here one can note a different behavior of Higgs bundles as opposed to vector bundles:
while in the latter case the condition that the class $\lambda(E)$ is nef is equivalent to curve semistability, 
in the Higgs case one needs the nefness of {\em all} classes $ \theta_d(\E)$; see \cite{BHR}
for an example of a rank 3 Higgs bundle on a curve with  $\theta_1(\E)$ nef,  $\theta_2(\E)$ not nef,
which is not semistable.

\subsection{Higgs numerical effectiveness}
In \cite{bruzzo-grana-adv} by means of the Higgs Grassmannians a notion of numerical effectiveness 
for Higgs bundles was introduced. It is a definition based on recursion on the rank of the
successive universal quotient bundles. Since we are going to use this definition later on,
we recall it here. 

\begin{defin}  \label{moddef} A Higgs bundle $\E=(E, \phi)$ of rank one on a smooth  projective variety is said
to be Higgs-numerically effective (for short, H-nef) if the underyling vector bundle $E$ is
numerically effective in the usual sense. If $\rk E \geq 2$ we
require that:
\begin{enumerate}\itemsep=-2pt \item all bundles $\gQ_{k}$ are Higgs numerically effective;
\item the line bundle $\det(E)$ is nef.
\end{enumerate}
If both $\E$ and $\E^\ast$ are Higgs-numerically effective,
$\E$ is said to be Higgs-numerically flat (H-nflat).
\end{defin}

\section{The conjecture in any rank }\label{giapponesi}
\subsection{A push-forward formula} We recall from \cite{Kaji-Tera} a push-forward formula for the Segre classes of the universal quotient bundle over Grassmann bundles  $\pi_d\colon\thegrass\to X$. Here $X$ will be a smooth projective variety of dimension $n$ and $E$ a rank $r>1$ vector bundle.  
Greek letters such as $\lambda$, $\mu$ will denote a {\em partition}, i.e., a {finite} nonincreasing sequence of natural numbers.
We let
$$ \vert \lambda \vert = \sum_{i=1}^q \lambda_i ,$$where $\lambda= (\lambda_1, \dots,\lambda_q)$, while $\pi_\ast$ will denote the push-forward of Chow groups $$ \pi_\ast \colon A^k(\thegrass) \to A^{k - d(r-d)}(X).$$
Moreover, we define the Segre classes of the vector bundle $F$ on a variety $X$ by the formula
$$  s(F)= \sum_{i=0}^{\dim X} (-1)^i s_i(F) = \frac1{c(F)}$$where $c(F)$ is the total Chern class of $F$ (we follow the normalization of \cite{Ful98}, hence the minus signs).
\begin{lemma}{\rm \cite{JLP,Kaji-Tera}}
Let $Q$ be the rank $d$ universal quotient bundle of a rank $r$ vector bundle $E$ over $X$. The following push-forward formula holds:
$$\pi_\ast \Delta_\mu(s(Q))=\Delta_{\mu-\varepsilon}(s(E)).$$
\end{lemma}
Here
\begin{enumerate}\itemsep=-2pt \item $\varepsilon$ is the partition of length $d$ whose elements are all $r-d$;
\item for every $c\in A^\bullet(X)$, $\Delta_\lambda(c)$ is the Schur polynomial  associated with $\lambda$ computed on the components 
of $c$ in $ A^\bullet(X)$, that is,
$$ \Delta_\lambda(c) = \det [ c_{\lambda_i+j-i}]_{1\le i,j \le n}.$$
\end{enumerate}
\begin{corol} \label{chi}Let $\chi = c_1(Q)$. For $N$ in the range $d(r-d) \le N \le d(r-d)+ n  $ one has
\begin{equation}\label{japanese} \pi_\ast \chi^N = \sum_{\vert\lambda\vert = N - d(r-d)} f^{\lambda+\varepsilon}\, \Delta_\lambda(s(E)) \end{equation}
where $f^\lambda$ is the number of standard Young tableaux of shape $\lambda$.\,\footnote{A standard Young tableau \cite{SYT} is a Young tableau
(say with $b$ boxes) whose boxes are labelled with the integers from 1 to $b$, in such a way that all rows and columns contain
increasing sequences of integers. To ``have shape $\lambda=(\lambda_1,\dots,\lambda_q)$'' means that the $i$-th row has $\lambda_i$ boxes.
Then $f^\lambda$ is the number of ways a Young tableau of shape $\lambda$ can be labelled.}

\end{corol}

{\bf Some notation and facts:}
\begin{itemize}\itemsep=-2pt \item  A natural number $k$ is regarded as a partition of length 1.  
For every $k$, $ \Delta_k(c) = c_k$, i.e., the degree $k$ term of $c$. In particular,
$$\Delta_k(c(E))=c_k(E), \qquad \Delta_k(s(E))=s_k(E).$$
\item Conjugate   partitions: given a partition $\lambda$, let $\bar\lambda$ be the conjugate partition, i.e., the partition which describes the conjugate Young tableau of
$\lambda$ (the one obtained by flipping it with respect to its diagonal.)
Then
$$\Delta_{\bar\lambda}(c(E)) = \Delta_\lambda(s(E)).$$
\item We shall denote by $p(k)$ the partition made up by $k$ 1's. Then $k$ and $p(k)$ are conjugate partitions, so that
$$\Delta_{p(k)}(s(E)) = c_k(E).$$
\item An explicit formula for $f^\lambda$ is the following \cite{Frob1900,MM1909}.
Let $\lambda$ be a partition of length $q>1$. Then 
\begin{equation} f^\lambda = \frac{\vert\lambda\vert !}{\prod_{i=1}^q\ell _i!} \prod_{1 \le i < j \le q} (\ell_i-\ell_j),\qquad \ell_i = \lambda_i + q - i 
\label{f}\end{equation}
\end{itemize}
We note that
\begin{equation}\label{g1}  f^{\varepsilon+1} = \frac{md}{r} f^\varepsilon, \qquad f^{\varepsilon+p(2)} = \frac{m(m+1)d(d-1)}{2r(r-1)}) f^\varepsilon \end{equation}
\begin{equation}\label{g2}   f^{\varepsilon+2} = \frac{m(m+1)d(d+1)}{2r(r+1)}  f^\varepsilon \end{equation}
where $m=d(r-d)+1$. So, setting $n=2$, Corollary \ref{chi} yields
\begin{equation}\label{f1} \pi _\ast \chi^{m-1}
= f^\varepsilon, \qquad \pi_\ast \chi^m=   \frac{md}{r} f^\varepsilon c_1(E)
\end{equation}
\begin{equation}\label{f2}  \pi_\ast\chi^{m+1}  = \frac{m(m+1)d(d+1)}{2r(r+1)}  f^\varepsilon  (c_1(E)^2-c_2(E)) +  \frac{m(m+1)d(d-1)}{2r(r-1)}) f^\varepsilon  c_2(E).\end{equation}
 
 \subsection{The result}\label{result}
 {In this Section we prove the main result of this paper.} 
 
{\begin{thm} Given a curve semistable Higgs bundle $\E=(E, \phi)$ on a surface $X$, if, for some $d$ in the range $0 < d < r=\rk E$, the Higgs Grassmannian $\grass{d}{\E}$ has an irreducible component $Z$ which is a divisor in $\thegrass$ and surjects onto $X$, then $\Delta(E)=0$.
\end{thm}} 

\begin{proof}Using the Leray-Hirsch Theorem we {define the classes $\beta_i \in A^i(X)$ by}
 $$ [Z ] = \chi \, \pi^\ast \beta_0  + \pi^\ast\beta_1 \in A^1(\thegrass),$$ {where $[Z]$ is the class of $Z$ in $A^1(Gr_d(E))$ (for a version of the Leray-Hirsch Theorem for Chow groups which applies to the present case see \cite{Kris}).}
Recalling equation \eqref{thetas},
a rather lengthy computation yields\footnote{From the second line we omit to write the pullbacks.
 We use the fact that integration on $\thegrass$ is the push-forward to $X$ followed by integration on $X$. The last line follows from the insertion of equations 
\eqref{f1}, \eqref{f2}  and some fractional calculus gimmickry.}
\begin{equation}\begin{aligned}   \int_Z \theta_d(\E)^ m &=  \int_{\thegrass}\theta_d(\E)^ m \smile\, [Z] = 
 \int_{\thegrass} [\chi - \tfrac{d}{r}\pi^\ast c_1(E)]^m (\chi\, \pi^\ast \beta_0     + \pi^\ast\beta_1) \\[5pt]
&= \int_{\thegrass} \left[ \beta_0\chi^{m+1} + \left(\beta_1-\beta_0 \frac{md}{r }c_1(E)\right)\chi^{m}  \right. + \\
& 
\left. \parbox{40mm}{\hfill} 
 \left(-\frac{md}r\beta_1\,c_1(E)+ \frac{m(m-1)d^2}{2r^2}\beta_0\,c_1(E)^2\right)\chi^{m-1} \right]\\
&= - \frac{m(m+1)d(r-d)}{r(r+1)(r-1)} f^\varepsilon \,\beta_ 0 \, \Delta(E)
= - \frac{m(m+1)(m-1)}{r(r+1)(r-1)} f^\varepsilon \,\beta_ 0 \, \Delta(E) .\label{delta} \end{aligned}\end{equation}
In the last line, by integrating over $X$, we   think of $\beta_0$ and $\Delta(E)$ as   integers. $\beta_0$ is positive by the following argument.  Denote by $\pi_Z$ the restriction of $\pi_d$ to $Z$, and by $Z_x$ its fiber at a point $x \in X$. We also denote $F_x = \pi_d^{-1}(x)$. We assume that  $\pi_Z$ is surjective so that by   \cite[Lemma 29.28.2]{Stacks}  every irreducible component of its fibers has either dimension $d(r-d)$ or $d(r-d)-1$
(to apply that result we need $Z$ to be integral but we can achieve that by replacing it with its reduced subscheme if needed).
{On the other hand, since $\pi_Z$ is proper, by the semicontinuity of the fiber dimension (see e.g.~\cite[Lemma 37.29.5]{Stacks}) the locus in $X$ where the fiber $Z_x$ has dimension $d(r-d)-1$ is open.
It is nonempty because if all fibers $Z_x$ have dimension $d(r-d)$ then $\grass{d}{\E}$ would coincide with $\thegrass$, a situation which we may exclude. So the generic 
fiber of $\pi_Z$ has dimension $d(r-d)-1$, the ``expected dimension''. Hence for generic $x$, $Z_x$ determines a class in $A^1(F_x)$.} Since the restriction $\chi_x$ of $\chi$ to $F_x$ is ample,
 $$ 0 < [Z_x] \cdot \chi_x^{m-2} = \beta_0\, \chi_x^{m-1}, $$ 
 so that  $\beta_0 > 0$.
So  we get
 $$\Delta(E) \le 0 \, .$$Now by the Bogomolov inequality $\Delta(E) \ge 0$ we obtain $\Delta(E) =0$.
\end{proof}

Note the ``miraculous disappearance'' of $\beta_1$!

\section{The conjecture in rank two} 
\subsection{A non emptiness result} The following result is a key to our proof of the Conjecture in rank two.

\begin{thm}  \label{nonempty}
Let $X$ be a smooth  curve, which may be projective or affine, and let $\E=(E, \phi)$ be a rank 2 Higgs bundle on $X$. 
The Higgs Grassmannian $\grass1{\E}$ of rank one Higgs quotients of $\E$ is  not empty.
\end{thm}

\begin{lemma}
Let $X$ be a smooth  curve (projective or affine), and let $\E=(E, \phi)$ be a rank 2 Higgs bundle on $X$.
Assume that the Higgs Grassmannian $\grass1{\E}$ of rank one Higgs quotients of $\E$ is  empty.
Then the Higgs field $\phi$ induces a splitting of the exact sequence
\begin{equation}\label{reldiff} 0 \to \pi^\ast_1  \Omega^1_X \to \Omega^1_{\P E}\to \Omega^1_{\P E/X}\to 0. \end{equation}
\end{lemma}
Note that since $\grass1{\E}$ is assumed to be  empty, the Higgs field $\phi$ is necessarily nonzero.
\begin{proof}
We refer to diagram \eqref{diagram}.
{Note that  $Q^\vee_1 \otimes S_1  \simeq  \Omega^1_{\P E/X}$. 
Define a map $s\colon  \Omega^1_{\P E/X} \to \Omega^1_{\P E} $ by letting
$$ s = ({b}_1 \circ \pi^\ast_1  \phi \circ {a}_1 ) \otimes \operatorname{id}_{Q_1 ^\vee}.$$
Since $\Omega^1_{\P E/X}$ is a line bundle, and $ \Omega^1_{\P E}$ is locally free, the morphism $s$ is either zero or is injective; but if it were zero, since the Higgs Grassmannian $\grass{1}{\E}$ is the zero locus of the composition $${b}_1  \circ \pi^\ast_1  \phi \circ {a}_1 =s \otimes \operatorname{id}_{Q_1 },$$ the Higgs Grassmannian $\grass{1}{\E}$ would be the entire $\P E$, and therefore would not be empty.
So we have an exact sequence
\begin{equation}\label{sost}
0 \to \Omega^1_{\P E/X} \xrightarrow{ s} \Omega^1_{\P E} \to R \to 0
\end{equation}where  $R$ is by definition the quotient, which has rank one. }

As the Higgs Grassmannian is empty, $s$ has no zeroes, so that $R$ is locally free. We form the diagram
\begin{equation}\label{diag2} \xymatrix{
0  \ar[r] &  \Omega^1_{\P E/X} \ar[r]^{ s}  & \Omega^1_{\P E}\ar@{=}[d] \ar[r]^r & R\ar[r] & 0 \\
0  \ar[r] & \pi^\ast_1  \Omega^1_X \ar[r]^{i}\ar@{..>}[u]_h \ar@{..>}[rru]^{\parbox{35mm}{\hfill}g} &  \Omega^1_{\P E}  \ar[r]^p& \Omega^1_{\P E/X} \ar[r] & 0 }
\end{equation}
{ We show that the morphism $g = r\circ i \colon \pi^\ast_1  \Omega^1_X \to R $ cannot be zero. Indeed if it were zero we would have a morphism $ h \colon \pi^\ast _1 \Omega^1_X \to  \Omega^1_{\P E/X}  $ which is not  zero as $i = s\circ h$. 
However since the fiber degree of   $\Omega^1_{\P E/X}$ is $-2$, the restriction of $h$ to each fiber of $\pi$ is zero,
i.e., $h=0$, which is a contradiction.}

{Thus, $g$ is nonzero, hence is injective. We prove it is an isomorphism. We have an exact sequence
$$ 0 \to \pi^\ast_1  \Omega^1_X \xrightarrow{g} R \to N \to 0 $$
where $N$ has rank zero. For any fiber $F$ of $\pi_1 $, by a standard argument, we have an exact sequence
$$ 0 \to \cO_F \to R_{\vert F} \to N_{\vert F} \to 0.$$
Since $R$ has fiber degree 0, $R_{\vert F}$ is isomorphic to  $\cO_F$, so that $N_{\vert F} =0$. As this holds for every fiber,
$N=0$, hence $g$ is an isomorphism.}

Now we have a diagram
$$\xymatrix{
& R \ar[d]_{g^{-1}} \\
0 \ar[r]  &  \pi^\ast \Omega^1_X  \ar[r]^i & \Omega^1_{\P E} \ar[ul]_r \ar[r] &  \Omega^1_{\P E/X} \ar[r] & 0 
}$$
which shows that the sequence \eqref{reldiff} splits.
\end{proof}

{\noindent \em Proof of Theorem \ref{nonempty}. } Note that the first line in diagram \eqref{diag2} splits as $i\circ g^{-1}$ is a section of $r$.
Let $t'$ be a retraction of the morphism $s$. Then $t= t'\otimes \operatorname{id}_Q$
is a morphism $Q_1 \otimes \Omega^1_{\P E/X} \to S_1$. Define $\zeta\colon \pi^\ast_1  E \to S_1 $ as 
$$ \zeta = t \circ {b}_1  \circ \pi^\ast_1  \phi.$$
Then 
$$  \zeta\circ {a}_1  =  t \circ {b}_1 \circ \pi^\ast_1 \phi \circ {a}_1  
= (t'\otimes \operatorname{id}_{Q _1} )\circ ( s \otimes \operatorname{id}_{Q^\vee_1}  )=  \operatorname{id}_{S_1 }$$
so that the first line in diagram \eqref{diagram} splits. But this is impossible as on each fiber of $\pi_1 $ that sequence
reduces to the Euler exact sequence. \qed

\begin{corol}\label{empty}
Let $\E=(E, \phi)$ be a rank two Higgs bundle on a smooth  $n$-dimensional projective variety $X$. The Higgs Grassmannian $\grass{1}{\E}$ has a component of dimension at least $n$ which surjects onto $X$.
\end{corol}

\begin{proof}
If $\grass{1}{\E}$ does not have such a component, let $Y$ be its image in $X$ (actually taking its reduced subscheme if it happens to be nonreduced),   let $C$ be a curve in $X$ not contained in $Y$, and let $C'$ be $C$ minus its intersection points with $Y$, and minus its possible singular points. Then $\E_{|C'}$ has an empty Higgs Grassmannian, a contradiction to Corollary \ref{nonempty}.
\end{proof}

\begin{rem}{The splitting of the exact sequence \eqref{reldiff}
means that $E$ is projectively flat,   i.e., $\P E$ comes from a  projective representation
$\pi_1(X) \to \operatorname{PGL}_2(\C)$ of the fundamental group of $X$. This agrees with a result in  \cite{PanZhangZhang}, whose authors, as a particular case of their equivalence of categories, prove that semistable Higgs bundles on a curve are  projectively flat. Note indeed that if the Higgs Grassmannian is empty, the Higgs bundle is stable.
}\end{rem}

\subsection{The proof}\label{conjrk2}We start with the case $\dim X = 2$, i.e., $X$ is smooth  projective surface. From Corollary \ref{empty} we get that $\grass{1}{\E}$ has a component $Z$ of dimension 2, and we are in the hypotheses of Section \ref{result}, so that $\Delta(E)=0$.

This can be extended to the higher dimensional case $\dim X=n$. Let $H$ be the class of an ample line bundle $L=\cO_X(D)$
and let $Y$ be the intersection of $n-2$ generic divisors in the linear system $\vert m D \vert$ for $m\gg 0$.    The {result for dimension 2} implies that $$\Delta (E)\cdot H^{n-2}=\tfrac1{m^{n-2}}\Delta(E_{|Y})=0.$$

So we have proved:
\begin{thm}\label{thm1}
Let $\E=(E, \phi)$ be a rank two Higgs bundle on an $n$-dimensional smooth projective variety. Then the following conditions are equivalent:
\begin{enumerate}\itemsep=-2pt
\item $\E$ is curve semistable;
\item $\E$ is semistable with respect to some polarization and the class $\theta_1(\E)$ is nef;
\item $\E$ is semistable with respect to a polarization $H$, and  $$\Delta (E)\cdot H^{n-2}=0.$$ 
\end{enumerate}
\end{thm}
\begin{proof}
We know from \cite{BHR} that (i) and (ii) are equivalent. The implication (ii) $\Rightarrow$ (iii) was just proved. To prove (iii) $\Rightarrow$ (ii) one can easily adapt the proof of one of the directions of Theorem 1.3 in \cite{BHR}.
\end{proof}

Actually condition (iii) can be strengthened to $\Delta(E)=0$. To see that we need some preliminary results. In particular we shall
prove that rank 2 H-nflat Higgs bundles have vanishing Chern classes {(H-nflat Higgs bundles were defined in Definition \ref{moddef})}.

\begin{lemma}\label{coche}
If $\E=(E, \phi)$ is an H-nflat Higgs bundle of any rank over a smooth $n$-dimensional  projective variety $X$ and $\Delta(E) \cdot H^{n-2}=0$, then all Chern classes of $E$ vanish.
\end{lemma}
\begin{proof} One has 
$c_1(E)=0$ as $\det(E)$ is numerically flat and then the condition on the discriminant implies $\operatorname{ch}_2(E) \cdot H^{n-2}=0$. By Theorem 2 in \cite{simpson-local} $\E$ has a filtration in (stable) Higgs bundles with zero Chern classes, whence the claim follows
(note that H-nflat Higgs bundles are semistable with respect to any polarization, see \cite{bruzzo-grana-adv}).
\end{proof}

\begin{thm}\label{thm2}
If $\E=(E, \phi)$ is a rank two H-nflat Higgs bundle over a smooth  projective surface $X$ then all Chern classes of $E$ vanish.
\end{thm}

\begin{proof}
Since H-nflat bundles are curve semistable   by Theorem \ref{thm1}, if $H$ is an ample class in $X$, we have $\Delta(E) \cdot H^{n-2}=0$. Then Lemma \ref{coche} implies the claim.
\end{proof}

We can now strengthen Theorem \ref{thm1} in the following form.

\begin{thm}\label{thm3}
Let $\E=(E, \phi)$ be a rank two Higgs bundle on an $n$-dimensional smooth projective variety. Then the following conditions are equivalent:
\begin{enumerate}\itemsep=-2pt
\item $\E$ is curve semistable;
\item $\E$ is semistable with respect to some polarization and the class $\theta_1(\E)$ is nef;
\item $\E$ is semistable with respect to some polarization and $\Delta (E)=0$.
\end{enumerate}
\end{thm}

\begin{proof}
We only need to show that if (ii) holds  then $\Delta(E)=0$. If $H$ is an ample class in $X$, from Theorem \ref{thm1} we have $\Delta(E) \cdot H^{n-2}=0$, which is equivalent to $\Delta(\End(E)) \cdot H^{n-2}=0$. So the Higgs bundle $\End(\E)$ is curve semistable, and since $c_1(\End(E))=0$, the Higgs bundle $\End(\E)$ is H-nflat \cite{bruzzo-grana-adv}. By Lemma \ref{coche} $\Delta(\End(E))=0$, that is, $\Delta(E)=0$.
\end{proof}

\label{conjanyrank}
Actually we have a stronger result.
\begin{thm}
Let $\E=(E, \phi)$ be a Higgs bundle on an $n$-dimensional smooth projective variety, such that the Higgs Grassmannian $\grass{d}{\E}$ has an irreducible component $Z$ which is a divisor in $\thegrass$ that surjects onto $X$. Then the following conditions are equivalent:
\begin{enumerate}\itemsep=-2pt
\item $\E$ is curve semistable;
\item  $\E$ is semistable   with respect to some polarization and  the class $\theta_d(\E)$ is nef;
\item $\E$ is semistable  with respect to some polarization and $\Delta (E)=0$.
\end{enumerate}
\end{thm}
Note that, contrary to the general situation, and due to the assumption that $\grass{d}{\E}$ has a divisorial component that surjects
onto $X$, the nefness of the class $\theta_d(\E)$ is enough to have curve semistability (but assuming a priori that   $\E$ is semistable).

\section{Inequalities for H-nef Higgs bundles} \label{inequalities}

We know from \cite{FulLaz,DPS} that a nef vector bundle $E$ on an $n$-dimensional smooth projective variety satisfies, for every ample class $H$ and for every $1 \le k \le n$,  the inequalities $$  s_k(E) \cdot H^{n-k} \ge 0,$$where $s_k(E)$ are the Segre classes of $E$. In this Section we prove a version of these inequalities for H-nef Higgs bundles.
\begin{thm}
Let $\E=(E, \phi)$ be a rank $r$ H-nef Higgs bundle over an $n$-dimensional smooth projective variety. Assume that $Z$ is an irreducible component of the Higgs Grassmannian $\grass{1}{\E}$,   
and write its   class in $A^{N}(\P E)$, where $N$ is the codimension of $Z$ in $\P E$, as $$[Z] = \sum_{i=0}^N \pi_1^\ast \beta_i \cdot \xi^{N-i},$$where $\beta_i \in A^{i}(X)$. Then for every $1 \le k \le n$ we have the inequality
\begin{equation}\label{Segre}\sum_{i=0}^{r-1} \beta_i \cdot s_{k-i}(E) \cdot H^{n-k}\ge 0
\end{equation}for any ample class $H$ in $X$. 
\end{thm}
 
\begin{proof}
We first prove the inequality for $k=n$. Since $\E$ is H-nef, the hyperplane class $\xi$ of $\P E$ is nef on $Z$. We shall use the identity \cite[\S3.1]{Ful98}
\begin{equation}\label{segre} \pi_{1\ast}\xi^{r+i-1}=s_i(E)  \end{equation} for $0 \le i \le n= \dim X$. Now we have 
\begin{eqnarray}
0 \le  \xi^{r+n-1-N} \frown[Z]&=&\int_{\P E}\xi^{r+n-1-N}\smile \sum_{i=0}^N \pi_1^\ast \beta_i \cdot \xi^{N-i}\\
&=&\sum_{i=0}^N \int_{\P E} \xi^{r+n-i-1}\smile \pi_1^\ast \beta_i
=\sum_{i=0}^{r-1}\beta_i \cdot s_{n-i}(E) 
\end{eqnarray}
For $k<n$ we have
\begin{eqnarray}
0 \le  (\pi_1^\ast H^{n-k} \smile \xi^{r+k-1-N}) \frown[Z]&=&  \sum_{i=0}^{r-1}\beta_i \cdot s_{k-i}(E) \cdot H^{n-k} . \end{eqnarray}
\end{proof}
Since $1\le N \le r-1$, the last summation in both equations may contain terms with $i>N$ but these are zero as $\beta_i=0$ in that range. 

\begin{rem} In the non-Higgs case we have $\beta_0=1$, $\beta_i=0$ for $i>0$ and we recover
the identities of \cite{FulLaz,DPS} when the Schur polynomial is a Segre class.
\end{rem}

 \par

\end{document}